\documentclass[12pt]{article}
\usepackage[left=3cm, right=3cm, top=2cm]{geometry}
\usepackage{amsmath, amsthm, amssymb, mathtools}
\allowdisplaybreaks
\usepackage{amscd,amssymb}
\usepackage[all]{xy}
\usepackage{color}
\usepackage{verbatim}
\usepackage{mathrsfs}
\usepackage{graphicx,epsfig}
\usepackage[T1]{fontenc}
\usepackage[ansinew]{inputenc}
\usepackage{float}
\usepackage{setspace}
\usepackage{xpatch}
\usepackage{amsfonts}
\usepackage[]{hyperref}
\usepackage{subcaption}
\usepackage{pst-node}
\usepackage{tikz-cd}
\usetikzlibrary{positioning}
\usetikzlibrary{shapes,arrows}
\usepackage{tikz}
\usepackage{graphicx,epsfig}
\usepackage{mathrsfs}
\usepackage{verbatim}
\allowdisplaybreaks
\usepackage[ansinew]{inputenc}
\usepackage[all]{xy}

\newtheorem{thm}{Theorem}[section]
\newtheorem{lem}[thm]{Lemma}

\newtheorem{Q}{Question}

\newtheorem{Def}[thm]{Definition}
\newtheorem{prop}[thm]{Proposition}
\newtheorem{claim}[thm]{Claim}

\def\A{\mathcal{A}}

\def\N{\mathcal{N}}
\def\F{\mathcal{F}}
\def\P{\mathcal{P}}
\def\Q{\mathcal{Q}}

\def\s{\mathcal{S}}

\newcommand{\si}{\sigma}
\newcommand{\sm}{\setminus}

\newcommand{\nin}{\notin}

\numberwithin{equation}{section}

\begin{document}
\title{Hom complexes of graphs of diameter $1$}
\author{Anurag Singh\footnote{{Chennai Mathematical Institute, India. Email: asinghiitg@gmail.com}}}
\date{}
\maketitle

\begin{abstract}
Given a finite simplicial complex $X$ and a connected graph $T$ of diameter $1$, in \cite{anton} Dochtermann had conjectured that $\text{Hom}(T,G_{1,X})$ is homotopy equivalent to $X$. Here, $G_{1,X}$ is the reflexive graph obtained by taking the $1$-skeleton of the first barycentric subdivision of $X$ and adding a loop at each vertex. This was proved by Dochtermann and Schultz in \cite{ds12}. 

In this article, we give an alternate proof of this result by understanding the structure of the cells of Hom$(K_n,G_{1,X})$, where $K_n$ is the complete graph on $n$ vertices. We prove that the neighborhood complex of $G_{1,X}$ is homotopy equivalent to $X$ and Hom$(K_n,G_{1,X})\simeq $ Hom$(K_{n-1},G_{1,X})$, for each $n\geq 3$.
\end{abstract}



\vspace{.1in}

\hrule


\section{{\bf Introduction}}

Lov\'asz \cite{Lov78} introduced the concept of $\text{Hom}$ complex (see Definition \ref{prel:homcompelex}) to obtain a lower bound for the chromatic number of graphs and thus proved the Kneser Conjecture. In $2006$, Babson and Kozlov (\cite{BK06} and \cite{BK07}), during their investigations of these complexes, proved that for any graph $G$, the Hom complex
$\text{Hom}(K_2,G)$ is homotopy equivalent to the Neighborhood Complex (see Definition \ref{prel:nbdcompelex}) of $G$.

For any loopless graph $G$, the Hom complex $\text{Hom}(K_2,G)$ has a free $\mathbb{Z}_2$-action. In 2007, Csorba \cite{cso07} proved that any free $\mathbb{Z}_2$-space can be realized (up to $\mathbb{Z}_2$-homotopy type) as $\text{Hom}(K_2,H)$ for some suitably chosen graph $H$. 

Consider a finite simplicial complex $X$. Let $G_{k,X}$ be the $1$-skeleton of the $k^{th}$ iterative barycentric subdivision of $X$, with a loop at every vertex. For a connected graph  $T$ with at least one edge, Dochtermann \cite{anton} in 2009, showed that $\text{Hom}(T,G_{k,X})$ is homotopy equivalent to $X$. In his proof, he had assumed that the integer $k$ satisfied the condition $2^{k-1}-1 \geq \text{diam}(T)$. Considering the situation when $k=1$, in \cite{anton} he conjectured the following result, which was later proved by Dochtermann and Schultz in \cite{ds12}.

\begin{thm}[{\cite[Conjecture $4.1$]{anton}}]\label{thm3} Let $X$ be a finite simplicial complex and $T$ be a connected graph with diameter $1$. Then, $\text{Hom}(T,G_{1,X}) \simeq X.$
\end{thm}

An alternate proof of this conjecture is provided in this article. Here, we study the complex Hom$(K_n,G_{1,X})$ in detail. The main results of this article are: 

 \begin{thm}\label{ng1x}
For any finite simplicial complex $X$, the neighborhood complex of $G_{1,X}$ is homotopy equivalent to $X$.
\end{thm}

 \begin{thm}\label{thm2}
Let $n \geq 2$. If $K_n$ is the complete graph on $n$ vertices and $X$ is any finite simplicial complex, then $\text{Hom}(K_n,G_{1,X}) \simeq Hom(K_2,G_{1,X}).$
\end{thm}


\section{{\bf Preliminaries}}

A {\it graph} is an ordered pair $G=(V(G),E(G))$, where $V(G)$ is the set of vertices and $E(G)  \subseteq V(G) \times V(G)$ is the set of edges of $G$. All the graphs in this article are assumed to be undirected {\it i.e.}, $(v,w)=(w,v)$. The vertices $v$ and $ w $ are said to be adjacent, if $(v, w)\in E(G)$. This is also denoted by $v \sim w$. If $(v,v) \in E(G)$ for all $v \in V(G)$, then the graph $G$ is called a {\it reflexive} graph. The {\it common neighborhood} of a subset $A$ of $V(G)$ is defined as $N(A)= \{x \in  V(G) \ | \ (x,a) \in E(G), \,\,\forall\,\, a \in A \}$.

A graph $H$ with $V(H) \subseteq V(G)$ and $E(H) \subseteq E(G)$ is called a {\it subgraph} of the graph $G$. For a nonempty subset $U$ of $V(G)$, the subgraph of  $G$ with vertex set $ U$ and edge set being $ \{(a, b) \in E(G) \ | \ a, b \in U\}$ is called the induced subgraph and is denoted by $G[U]$.

A {\it path} in $G$ from $v$ to $w$ is a sequence of distinct vertices  $v_0, \ldots , v_{k}$ such that $v_0=v$, $v_{k}=w$ and $(v_i, v_{i+1}) \in E(G)$, $\forall \ i \in \{0, 1, \ldots , k-1\}$. A graph is said to be {\it connected}, if there exists a path between any two distinct vertices of $G$.  The {\it distance} from $u$ to $v$, denoted by $d(u,v)$, in a connected graph $G$ is the minimum number of vertices required for a path from $u$ to $v$. For a connected graph $G$, the {\it diameter} of $G$, denoted by $\text{diam} (G) $, is $ \text{max} \{d(v, w) \ | \ v, w \in V(G)\}$.

  A {\it graph homomorphism} from  $G$ to $H$ is a function $f : V(G) \to V(H)$ where $(v,w) \in E(G)$ implies that $(f(v),f(w)) \in E(H).$

 An {\it (abstract) simplicial complex} on a vertex set $V$ is a collection $X$ of subsets of $V$ such that if $\si \in X$ and $\tau \subset \si$, then $\tau \in X$. The elements of $X$ of cardinality $k+1$ are referred to as {\it k-simplices} or $k$-dimensional simplices.  A simplex $\si \in X$ such that $\si \not\subset \tilde{\si}$ for any $\tilde{\si} \in X$, is called a {\it facet}.
 
 The {\it $k$-skeleton} of the simplicial complex $X$, denoted by $X^{k}$, is the collection all those simplices $\si$, with $\text{dim}(\si) \leq k$. Here $\text{dim} (\si)$ denotes the dimension of $\si$. 

\begin{Def}\label{def:Cl}
The {\it clique complex}, Cl$(G)$ of a graph $G$ is the simplicial complex with $A \in $ Cl$(G)$ if and only if the induced subgraph $G[A]$ is complete.
\end{Def}

 We now define a simplicial collapse.
 
 \begin{Def}
 If $X$ is an abstract simplicial complex and $\sigma, \tau \in X$ such that
 
 \begin{itemize}
 \item[(i)] $  \tau \subsetneq \sigma$, and
 \item[(ii)] $\sigma$ is the only facet of $X$ containing $\tau$,
 \end{itemize}
 
then $\tau$ is called a {\it free face} of $\si$ and $(\si, \tau)$ is called a {\it collapsible pair}.  Further, a {\it simplicial collapse} of $X$ is the simplicial complex $Y$ obtained by the removal of all simplices $\gamma$ where $\tau\subseteq \gamma \subseteq \si$. This is denoted by $X \searrow Y$. Additionally, if dim$(\tau)= \text{dim}(\si)-1$, then this is referred to as an {\it elementary collapse}.
 \end{Def}

The poset whose elements are the simplices of $X$, ordered by inclusion, is called the {\it face poset} of $X$ and is denoted by $\F(X)$.
For any poset $P$, the simplicial complex whose simplices are the chains of $P$ is called the {\it order complex} of $P$ and is denoted by $\Delta(P)$. 

The  {\it first barycentric subdivision} of a simplicial complex $X$, denoted by $\text{Sd}(X)$, is the order complex of the face poset of $X$, {\it i.e.}, $\text{Sd}(X)=\Delta(\F(X))$. In other words, $\text{Sd}(X)$ is a simplicial complex whose vertices are all the simplices of $X$. Two vertices in $Sd(X)$ are in a face, if and only if one of them is contained in the other. Thus, the facets of $\text{Sd}(X)$ are the maximal chains in $X$, considered as a poset. If $k\geq 2$, then the $k^{th}$  iterative barycentric subdivision of $X$ is the barycentric subdivision of $\text{Sd}^{k-1}(X)$.

  For a simplicial complex $X$, $G_{k,X}$ is a reflexive graph which is obtained by taking the $1$-skeleton of the $k^{th}$ barycentric subdivision of $X$ and adding a loop at each vertex.
  
  {\bf Example:}
  \begin{figure}[H]
	\begin{subfigure}[]{0.45 \textwidth}
		\centering
		\begin{tikzpicture}
		[scale=0.6]
\draw [line width=1.2pt] (2.5,5.)-- (0,0);
\draw [line width=1.2pt] (5,0)-- (0,0);
\draw [line width=1.2pt] (2.5,5.)-- (5,0);
\draw (-0.8,0) node[anchor=north west] {$2$};
\draw (5.1,0) node[anchor=north west] {$3$};
\draw (2.1,5.99) node[anchor=north west] {$1$};
\end{tikzpicture}
		\caption{$X$}\label{fig:X}
	\end{subfigure}
	\begin{subfigure}[]{0.45 \textwidth}
		\centering
		\begin{tikzpicture}
		[scale=0.5, vertices/.style={circle, draw = black!100, fill= gray!0, inner sep=0.5pt}] 
		\tikzset{every loop/.style={looseness=10}}
		\node (a) at (2.5,5) {$\{1\}$};
		\node (b) at (0,0)  {$\{2\}$};
		\node (c) at (5,0)  {$\{3\}$};
		\node (d) at (1.25,2.5)  {$\{1,2\}$};
		\node (e) at (2.5,0)  {$\{2,3\}$};
		\node (f) at (3.75,2.5)  {$\{1,3\}$};
		
		\foreach \from/\to in {a/d, d/b, b/e, e/c, c/f, f/a}
		\draw (\from) -- (\to);
		\path 
		(a) edge[loop above] (a)
		(c) edge[loop right] (c)
		(f) edge[loop right] (f)
		(e) edge[loop below] (e)
		(b) edge[loop left] (b)
		(d) edge[loop left] (d);
		\end{tikzpicture}
		\caption{$G_{1,X}$}\label{fig:G_{1,X}}
	\end{subfigure}
	\caption{}\label{fig:totalX}
\end{figure}
  
 A {\it prodsimplicial complex} is a polyhedral complex, each of whose cells is a direct product of simplices (refer \cite{Koz07} for more details on prodsimplicial complexes).

\begin{Def}\label{prel:nbdcompelex}The {\it neighborhood complex}, $\N(G)$ of a graph $G$ is the abstract simplicial complex with vertices being the non isolated vertices of $G$ and simplices being all those subsets of $V(G)$ with a common neighbor. 
\end{Def}

\begin{Def}\label{prel:homcompelex}
Given two graphs $G$ and $H$, $Hom(G,H)$ is a poset whose elements are all functions $\eta :V(G) \longrightarrow 2^{V(H)}\sm \{\emptyset\}$, such that 
\begin{equation}\label{homcondition}
(u,v) \in E(G) \implies \eta(u)\times \eta(v) \subseteq E(H).
\end{equation} The partial order is given by inclusion, {\it i.e.}, $\eta \leq \eta^\prime$ whenever $\eta(u) \subseteq \eta^\prime(u)$ for all $u \in V(G)$.
\end{Def}

The $\text{Hom}$ complex is often thought of as a topological space. In this context, we refer to the space obtained as the geometric realization of the order complex of the poset $\text{Hom}(G,H)$. We now discuss some results related to Hom complexes which are used in this article.

\begin{lem}[{\cite[Proposition $4.2$]{BK06}}]\label{homtonbd}
Hom$(K_2,G)$ is homotopy equivalent to $\N(G)$.
\end{lem}
  
  Let $x$, $y$ be distinct vertices of the graph $G$ such that $N(x) \subseteq N(y).$ The subgraph $G'$, where $V(G')= V(G) \setminus \{x\}$ and 
$E(G') = E(G) \setminus \{(v, w) \in E(G) \ | \ \text{either } v=x \text{ or } w = x\}$, is called a {\it fold} of $G$ and denoted by $G \setminus x$.
  In $2006$, Babson and Kozlov proved that the folds in the first coordinate of the $\text{Hom}$ complexes preserve the homotopy type. 

\begin{thm}[{\cite[Proposition $5.1$]{BK06}}]\label{fold}
If $G$ and $H$ are graphs, and $u$ and $v$ are vertices of $G$ such that
$N(u)\subseteq N(v)$, then $\text{Hom}(G\sm u,H)\simeq Hom(G,H)$.
\end{thm}

In fact, Kozlov proved that $\text{Hom}(G\sm u,H)$ is a strong deformation retract of $\text{Hom}(G,H)$. The following lemma, often called a {\it Quillen type Lemma}, proved by  Babson and Kozlov has been used in this article.
\begin{lem}[{\cite[Proposition $3.2$]{BK06}}]\label{fibthm}
Let $\varphi : P \longrightarrow Q$ be a map of finite posets. If $\varphi$ satisfies
\begin{enumerate}
\item[(A)] $\Delta(\varphi^{-1}(q))$ is contractible for each $q \in Q$ and
\item[(B)] for each $p \in P$ and $q \in Q$ with $\varphi(p) \geq q$, the poset $\varphi^{-1}(q)\cap P_{\leq p}$ has a  maximal element,
\end{enumerate}
then $\Delta(\varphi): \Delta(P) \longrightarrow \Delta(Q)$ is a homotopy equivalence.
\end{lem}

We now state the {\it Nerve Lemma}, a useful tool from Topological Combinatorics, which has been used in this article. We first define the nerve associated to a cover of a topological space.

\begin{Def} Let $X$ be a topological space with a cover $\A=\{X_{i}\}_{i\in I}$, {\it i.e.}, $X_i \subseteq X$ for each $i\in I$ and $\bigcup_{i\in I}X_i=X$. The  nerve of $\A$, denoted by $\text{ Nerve}(\A)$, is  the simplicial complex with vertex set $I$ and where $\si =\{x_{1}, \ldots , x_{k}\} \subseteq I$ is a simplex in $\text{ Nerve}(\A)$ if and only if $X_{x_1}\cap X_{x_2} \cap \ldots \cap X_{x_n}$ is nonempty.
\end{Def}

\begin{thm}[{\cite[Theorem $2.13$]{cso05}}]\label{nerve}
Let $X$ be a simplicial complex and $\{X_i \ | \ i \in I\}$ be a family of subcomplexes such that $X = \bigcup_{i \in I} X_i$. If every nonempty finite intersection $X_{i_1}\cap \ldots \cap X_{i_r}$ is contractible, then $X$ and $Nerve(\{X_i \ | \ i \in I \})$ are homotopy equivalent.
\end{thm} 
 
\section{Homotopy type of \texorpdfstring{$\N(G_{1,X})$}{NS}} 

Let $X$ be a finite simplicial complex and let $G_{1,X}$ be the graph given by the $1$-skeleton of the first barycentric subdivision of $X$ with a loop at each vertex.  Thus, $V(G_{1,X})$ is the set of simplices of $X$.  Clearly, $(v,w) \in E(G_{1,X})$ implies that either $v \subseteq w$ or $w \subseteq v$.  

For each $x _{i} \in X^{0}$, where $X^{0}$ is the $0$-skeleton of $X$, define
 \begin{equation}
 X_{x_i} = \P(N(x_i)) \setminus \{\emptyset\},
 \end{equation} 
 where $\P(N(x_i))$ is the {\it power set} of $N(x_i)=\{y \in V(G_{1,X}) \ | \ (x_i,y) \in E(G_{1,X})\}$. Thus, $\si \in X_{x_i}$ implies that $\si \cap X^{0} =\emptyset$  or $\si \cap X^{0} =\{x_{i}\}.$
  
\begin{prop}\label{prop1} $\bigcup \{ X_{x_i} \ | \ x_i \in X^{0}\}$ is homotopy equivalent to $X$.

\end{prop}
 \begin{proof} Let $\A=\{X_{x_i} \ | \ x_i \in X^{0}\}.$ We first show that $\text{Nerve} (\A)$ is  $X$.
 
 Let $\si=\{y_1, \ldots , y_{k}\} \in X$. Here, $\si$ is a vertex in $\text{Sd}(X)$ and thus $\si \in V(G_{1,X})$. Since $\{y_i\} \subseteq \si$, $\forall \ i \in \{1, \ldots , k\}$, we see that in $\text{Sd}(X)$, $\{y_{i}, \si\}$ is a $1$-simplex. Thus, $\si \in N(y_i)$, $\forall \ i \in \{1, \ldots , k\}$. This implies that $\si \in X_{y_1} \cap \ldots \cap X_{y_k}$, {\it i.e.}, $\si \in \text{Nerve} (\A)$.
 
Conversely, consider a simplex $\tau \in \text{Nerve} (\A)$. Here, there exist $y_1, \ldots , y_\ell \in X^0$, such that $ X_{y_1} \cap X_{y_2} \ldots \cap X_{y_\ell}\neq \emptyset$ and $\tau =\{y_{1}, \ldots , y_{\ell}\}$. So, there exists $y \in V(G_{1,X})$ such that $y \in X_{y_1}  \cap X_{y_2} \cap \ldots \cap X_{y_\ell}$, {\it i.e.,} $y \in N(y_i)$ $\forall~ i  \in \{1,\ldots,\ell \}$. In other words, $(y,y_i) \in E(G_{1,X})$ $\forall~ i  \in \{1,\ldots,\ell \}$. Now, $y \in V(G_{1,X})$ implies that $y \in X$. As $X$ is a simplicial complex, $y \in X$ and $y_i \in y$, $\forall~ i  \in \{1,\ldots,\ell \}$, we get $\{y_1,\ldots,y_\ell\}\subseteq y$. This implies that $\{y_1,\ldots,y_\ell\} \in X$. Thus, $\tau \in X$. Therefore, Nerve${(\A)}=X$.
 
Each of the $X_{x_i}$'s is contractible. Further every nonempty intersection of the $X_{x_i}$'s, $x_i \in X^{0}$ is a simplex and therefore contractible. Using the Nerve Lemma (Theorem \ref{nerve}), $\text{Nerve}(\A)$ is homotopy equivalent to $\bigcup\{ X_{x_i} \ | \ x_i \in X^{0}\}$.  As $\text{Nerve}(\A) =X$, Proposition \ref{prop1} follows.
\end{proof}

We now label the elements of $V(G_{1,X})$ as $V(G_{1,X})=\{\si_1,\si_2,\ldots, \si_p\}$, where $\text{dim}(\si_i)\geq \text{dim}(\si_{i+1})$ for $i \in \{1, \ldots ,p-1\}$. If the cardinality of $X^{0}$ is $q$, then the last $q$ elements of $V(G_{1,X})$ are from $X^{0}$.  Assume that $p-q \geq 1$.
For each $\ell \in \{1, \ldots , p-q+1\}$, define
\begin{equation}\label{defofkl}
K_\ell = \{\si  \ | \ \si \subseteq N(\si_i) \ \text{and} \ \ell \leq i \leq p \}. 
\end{equation}
{\bf Example:} Let $X$ be a simplicial complex as depicted in Figure \ref{fig:totalX}(a) ({\it i.e.}, boundary of a $2$ simplex) and $V(G_{1,X})= \{\{1,2\},\{1,3\},\{2,3\},\{1\},\{2\},\{3\}\}$ be a labeling on vertices of $G_{1,X}$. Then $K_1$ and $K_2$ are given in Figure \ref{fig:Kl}(a) and \ref{fig:Kl}(b) respectively.
\begin{figure}[H]
\begin{subfigure}[]{0.45 \textwidth}
\centering
\begin{tikzpicture}
\filldraw[fill=gray!60, draw=black] (6.,5.)--(4,4)--(4,2)--cycle;
\filldraw[fill=gray!60, draw=black] (6.,1.)--(8,2)--(4,2)--cycle;
\filldraw[fill=gray!60, draw=black] (6.,5.)--(8,2)--(8,4)--cycle;
\draw [line width=1.2pt] (6.,1.)-- (8.,2.);
\draw [line width=1.2pt] (8.,2.)-- (8.,4.);
\draw [line width=1.2pt] (8.,4.)-- (6.,5.);
\draw [line width=1.2pt] (6.,5.)-- (4.,4.);
\draw [line width=1.2pt] (4.,4.)-- (4.,2.);
\draw [line width=1.2pt] (4.,2.)-- (6.,1.);
\draw [line width=1.2pt] (6.,5.)-- (8.,2.);
\draw [line width=1.2pt] (8.,2.)-- (4.,2.);
\draw [line width=1.2pt] (4.,2.)-- (6.,5.);
\draw [line width=1.2pt] (4.,4.)-- (6.,1.);
\draw [line width=1.2pt] (6.,1.)-- (8.,4.);
\draw [line width=1.2pt] (8.,4.)-- (4.,4.);
\filldraw[fill=gray!60, draw=black] (6.,5.)--(4,4)--(8,4)--cycle;
\filldraw[fill=gray!60, draw=black] (6.,1.)--(4,4)--(4,2)--cycle;
\filldraw[fill=gray!60, draw=black] (6.,1.)--(8,2)--(8,4)--cycle;
\draw [dotted, line width=1.2pt] (5.35574,4.)-- (6.,5.);
\draw [dotted, line width=1.2pt] (4.,2.)-- (4.634565,2.98);
\draw [dotted, line width=1.2pt] (5.3004,2.0)-- (4.,2.);
\draw [dotted, line width=1.2pt] (6.7,2.0)-- (8.,2.);
\draw [dotted, line width=1.2pt] (6.6537,4.)-- (6.,5.);
\draw [dotted, line width=1.2pt] (7.345,2.98)-- (8.,2.);
\draw (5.5308698365027745,0.843431837623781) node[anchor=north west] {$\{2,3\}$};
\draw (3.066306132666899,2.252933592781262) node[anchor=north west] {$\{2\}$};
\draw (8.115433540338649,2.252933592781262) node[anchor=north west] {$\{3\}$};
\draw (8.14896034050903,4.403057490138038) node[anchor=north west] {$\{1,3\}$};
\draw (5.5106207140815165,5.608783322703882) node[anchor=north west] {$\{1\}$};
\draw (2.866306132666899,4.403057490138038) node[anchor=north west] {$\{1,2\}$};
\end{tikzpicture}
\caption{$K_1$}
\label{fig:K1}
\end{subfigure}
\begin{subfigure}[]{0.45 \textwidth}
\centering
\begin{tikzpicture}
\filldraw[fill=gray!60, draw=black] (6.,1.)--(8,2)--(4,2)--cycle;
\filldraw[fill=gray!60, draw=black] (6.,5.)--(8,2)--(8,4)--cycle;
\draw [line width=1.2pt] (6.,1.)-- (8.,2.);
\draw [line width=1.2pt] (8.,2.)-- (8.,4.);
\draw [line width=1.2pt] (8.,4.)-- (6.,5.);
\draw [line width=1.2pt] (6.,5.)-- (4.,4.);
\draw [line width=1.2pt] (4.,4.)-- (4.,2.);
\draw [line width=1.2pt] (4.,2.)-- (6.,1.);
\draw [line width=1.2pt] (6.,5.)-- (8.,2.);
\draw [line width=1.2pt] (8.,2.)-- (4.,2.);
\draw [line width=1.2pt] (4.,4.)-- (6.,1.);
\draw [line width=1.2pt] (6.,1.)-- (8.,4.);
\draw [line width=1.2pt] (8.,4.)-- (4.,4.);
\filldraw[fill=gray!60, draw=black] (6.,5.)--(4,4)--(8,4)--cycle;
\filldraw[fill=gray!60, draw=black] (6.,1.)--(4,4)--(4,2)--cycle;
\filldraw[fill=gray!60, draw=black] (6.,1.)--(8,2)--(8,4)--cycle;
\draw [dotted, line width=1.2pt] (5.3004,2.0)-- (4.,2.);
\draw [dotted, line width=1.2pt] (6.7,2.0)-- (8.,2.);
\draw [dotted, line width=1.2pt] (6.6537,4.)-- (6.,5.);
\draw [dotted, line width=1.2pt] (7.345,2.98)-- (8.,2.);
\draw (5.5308698365027745,0.843431837623781) node[anchor=north west] {$\{2,3\}$};
\draw (3.066306132666899,2.252933592781262) node[anchor=north west] {$\{2\}$};
\draw (8.115433540338649,2.252933592781262) node[anchor=north west] {$\{3\}$};
\draw (8.14896034050903,4.403057490138038) node[anchor=north west] {$\{1,3\}$};
\draw (5.5106207140815165,5.608783322703882) node[anchor=north west] {$\{1\}$};
\draw (2.866306132666899,4.403057490138038) node[anchor=north west] {$\{1,2\}$};
\end{tikzpicture}
\caption{$K_2$}\label{fig:K2}
\end{subfigure}
\caption{}\label{fig:Kl}
\end{figure}

The following are easy observations.
\begin{enumerate}
\item $K_{\ell+1}$ is a subcomplex of $K_\ell$, 
\item $K_1=\N(G_{1,X})$ and 
\item $K_{p-q+1}=K_{X^0}$, where $K_{X^{0}} =\{\si \ | \ \si \subseteq N(x),~ \text{where} ~ x \in X^{0}\}$. 
\end{enumerate}
Recall that for any $\si \in V(G_{1,X})$, $N(\si)=\{\tau \in V(G_{1,X}) \ | \ \text{either}~ \tau \subseteq \si ~  \text{or} ~ \si \subseteq \tau\}$. In $V(G_{1,X})$, if $\si \neq \tau$ and $\text{dim} (\si) =\text{ dim} (\tau)$, then both $\si \nin N(\tau)$ and $\tau\nin N(\si)$ hold. 

\begin{prop}\label{prop6.2.2}
$K_{1} \searrow K_{2}$.
\end{prop}
\begin{proof} 
As $(\si_{i}, \si_{i}) \in E(G_{1,X})$ for all $\si_{i} \in V(G_{1,X})$, we get $ \si_{1} \in N(\si_{1})$. Clearly, $(N(\si_1), N(\si_{1}) \sm \{\si_1\})$ is a collapsible pair. Let $K_{1}^{1}$ be the subcomplex of $K_{1}$ obtained from this collapse. If possible, consider a simplex $\sigma \in K_{1}^{1}$ such that $\si_{1} \in \si$. If $\si \sm \{\si_{1}\}$ is a free face of $\si$, then let $K_{1}^{2}$ be the subcomplex of $K_{1}^{1}$ obtained by removing the collapsible pair $(\si, \si \sm \{\si_{1}\})$. Repeat this process to obtain a subcomplex  $\widetilde{K_{1}}$ of $K_{1}$, where $\si' \in \widetilde{K_{1}}$ and $\si_{1} \in \si'$ implies that $\si' \sm \{\si_{1}\}$ is not a free face of $\si'$. Clearly, $K_2$ (as defined in Equation \eqref{defofkl}) is a subcomplex of $\widetilde{K_1}$. 

\begin{claim}\label{c1}
$\widetilde{K_1}=K_2.$
\end{claim}
Suppose that $\widetilde{K_1} \sm K_{2} \neq \emptyset .$ There exists $\tau \in \widetilde{K_1} \sm K_{2}$  such that $\si_{1} \in \tau$. Further, by definition of $\widetilde{K_1}$, $\tau \sm \{\si_1\}$ is not a free face of $\tau$, {\it i.e.},  there exists $i_0 \in \{2, \ldots , p\}$ such that $\tau \sm \{\si_{1}\} \subseteq N(\si_{i_0})$ and $\text{dim}(\si_{i_0}) \leq  \text{dim} (\si_{1})$. 
Observe that $\tau \sm \{\si_{1} \} \subseteq N(\si_{i_0})$ implies that $\forall  \  \widetilde{\tau} \in \tau \sm \{ \si_{1}\}$, either  $\widetilde{\tau} \subseteq \si_{i_0}$ or $ \si_{i_0} \subseteq \widetilde{\tau}.$  Further, as $\widetilde{\tau} \subsetneq \si_{1}$, we see that
\begin{enumerate}  
\item $\si_{1} \cap \si_{i_0} \neq \emptyset$ and
\item $\widetilde{\tau} \subseteq \si_{i_0} \Rightarrow \widetilde{\tau} \subseteq \si_{i_0} \cap \si_{1}$ and $ \si_{i_0}  \subseteq \widetilde{\tau} \Rightarrow \si_{i_0} \cap \si_{1}  \subseteq \widetilde{\tau}.$
\end{enumerate}
 
Thus, $\tau \sm \{\si_{1}\} \subseteq N(\si_{1} \cap \si_{i_0}).$ Also, $\si_{1} \cap \si_{i_0} \subsetneq \si_{1}$ implies that  $\text{dim}(\si_{1} \cap \si_{i_0} ) < \text{dim}(\si_{1})$. As both $ \si_{1} \in N(\si_1 \cap \si_{i_0})$ and $\tau \sm \{\si_{1}\} \subseteq N(\si_{1} \cap \si_{i_0})$ hold, we get $\tau \subseteq N(\si_{1} \cap \si_{i_0}) \in K_{2}$, a contradiction to the assumption that $\tau \in \widetilde{K_1} \sm K_{2}$.
Therefore, $\widetilde{K_1} =K_2$, thereby proving Claim \ref{c1}. This completes the proof of Proposition \ref{prop6.2.2}.
\end{proof}

By an argument similar to the one above, $K_\ell$ collapses onto $K_{\ell+1}$, for $ \ell\in \{1, \ldots , p-(q+1)\}.$ We now prove:

\begin{prop}\label{prop3}
 $K_{p-q}$ collapses onto $K_{p-q+1}$.
\end{prop}
\begin{proof} By definition of $G_{1,X}$, $\si_{p-q} \in N(\si_{p-q})$. Initially we prove by contradiction that   $(N(\si_{p-q}), N(\si_{p-q}) \sm \{\si_{p-q}\})$ is a collapsible pair.

 If $(N(\si_{p-q}), N(\si_{p-q}) \sm \{\si_{p-q}\})$ is not a collapsible pair., then there must exist $r \in \{p-q+1, \ldots , p\}$, such that $N(\si_{p-q}) \sm \{\si_{p-q}\} \subseteq N(\si_r)$. 
 Using Equation \eqref{defofkl} and observation $3$, we note that $\text{dim}(\si_{p-q}) =1 > \text{dim}(\si_{r})=0$.
 
Since $\si_{p-q} \in V(G_{1,X})$, $\si_{p-q} \subseteq X^0$. Thus, as $\text{dim}(\si_{p-q}) >0$, each $x \in \si_{p-q}$ also belongs to $N(\si_{p-q}) \sm \{\si_{p-q}\}$. In particular $x \in N(\si_{r})$, {\it i.e.}, $x \in \si_{r}$ for all $x \in \si_{p-q}$. In other words, $\si_{p-q}  \subseteq \si_r$. This implies that $\text{dim}(\si_{p-q}) \leq  \text{dim}(\si_{r})=0$,  a contradiction. Therefore,  $(N(\si_{p-q}), N(\si_{p-q}) \sm \{\si_{p-q}\})$ is a collapsible pair. As in the case of the proof of Proposition \ref{prop6.2.2}, let $\widetilde{K_{p-q}}$ be the subcomplex of $K_{p-q}$, such that whenever $\si \in \widetilde{ K_{p-q}}$ and $\si _{p-q} \in \si$, the simplex $\si \sm \{\si_{p-q}\}$ is not a free face of $\si$.  
Clearly, $\widetilde{ K_{p-q}}$ is a subcomplex of $K_{p-q+1}.$ 

 Assume that there exists $\tau \in \widetilde{ K_{p-q}} \sm K_{p-q+1}$.
Here, $\tau \subseteq N(\si_{p-q})$ and $\tau \not\subset N(\si_{i})$, $\forall \ i \in \{p-q+1, \ldots , p\}$. Therefore,
there exists $j_0 \in \{p-q+1, \ldots, p\}$ such that $\tau \sm \{\si_{p-q}\} \subseteq N(\si_{j_0})$ and $\text{dim}(\si_{p-q}) > \text{dim}(\si_{j_0})$.

 If $\si_{p-q} \cap \si_{j_0} \neq \emptyset$, then by an argument similar to that in Claim \ref{c1}, we see that $\tau \in K_{p-q+1}$, a contradiction to the assumption that $\tau \in \widetilde{ K_{p-q}} \sm K_{p-q+1}$..

Now consider $\si_{p-q}\cap \si_{j_0}= \emptyset$. If there exists $x_0\in \si_{p-q}$ such that $x _0 \in \widetilde{\tau_i}$ for each $\widetilde{\tau}
 \in \tau \sm \{\si_{p-q}\}$, then $\tau \sm \{\si_{p-q}\} \subseteq N(x_0)$. Further $\si_{p-1} \subseteq N(x_0)$ implies that $\tau \subseteq N(x_0)$. This implies that $\tau \in K_{p-q+1}$, again a contradiction. Thus, there exists at least one element $\widetilde{\tau_{1}}$ in $\tau \sm \{\si_{p-q}\}$, where $x_0 \nin \widetilde{\tau_{1}}.$  Hence, $\si_{p-q} \not\subseteq \widetilde{\tau_{1}}.$ 
 
Since $\widetilde{\tau_{1}} \in N(\si_{p-q})$, we have $\widetilde{\tau_{1}} \subsetneq \si_{p-q}$. Further, 
 $\si_{p-q}\cap \si_{j_0}= \emptyset$ implies that $\si_{j_0} \not\subseteq \si_{p-q}$. We thus conclude that $\widetilde{\tau_{1}} \in N(\si_{j_0})$. Therefore $\widetilde{\tau_{1}} \subsetneq \si_{j_0}$. This implies that $\widetilde{\tau_{1}} \in \si_{p-q} \cap \si_{j_0}$, a contradiction to the assumption that $\si_{p-q}\cap \si_{j_0}= \emptyset$. Therefore,  $\widetilde{ K_{p-q}} = K_{p-q+1}$. This completes the proof of Proposition \ref{prop3}.
\end{proof}

We can now determine the homotopy type of $\mathcal{N}(G_{1,X})$.

\begin{proof}[Proof of Theorem \ref{ng1x}:] ~~~~~~~

\medskip

The facets of $\mathcal{N}(G_{1,X})$ are of the form $N(\si_i)\subseteq V(G_{1,X})$. By Propositions \ref{prop6.2.2} and  \ref{prop3}, $K_1$ collapses onto $K_{p-q+1}$.  As  $K_{p-q+1}= K_{X^{0}}$, using Proposition \ref{prop1}, we see that $K_{p-q+1} \simeq X$. Since $K_1=\mathcal{N}(G_{1,X})$ and $K_1 \simeq K_{p-q+1}$, we see that $\N(G_{1,X}) \simeq X$, thereby completing the proof of Theorem \ref{ng1x}.
\end{proof}


\section{Homotopy type of \texorpdfstring{Hom$(T,G_{1,X})$}{NS}}

Let $T$ be a finite connected graph with $|V(T)| \geq 2$ and $\text{diam}(T)=1$. As the diameter of $T$ is $1$, we conclude that any two distinct vertices of $T$ are connected by an edge, {\it i.e.},
\begin{equation}\label{diamcond}
\forall ~x, y \in V(T),~ x\neq y \Longrightarrow (x,y)\in E(T).
\end{equation}
We consider the following two cases: in the first we assume that there is at least one vertex in $T$ with a loop and in the second that there are no loops at any vertex in $T$. 

\begin{enumerate}
\item There exists at least one vertex in $T$, say $ v$, such that $(v,v) \in E(T)$. 

Let $V(T) =\{u_0,u_1,\ldots, u_s,v\}$. By Equation \eqref{diamcond}, $N(u_0)\subseteq N(v)$. Therefore, $T$ folds onto a subgraph $T\setminus u_0$. Using Proposition  \ref{fold}, we get $\text{Hom}(T,G_{1,X})\simeq \text{Hom}(T\setminus u_0,G_{1,X})$. Again in $T\setminus u_0$, $N(u_1)\subseteq N(v)$ and therefore $T\setminus u_0$ folds onto a subgraph $T\setminus \{u_0,u_1\}$. Hence, by a sequence of folds, we get that $T$ folds onto the subgraph $T'$, where $V(T')=\{v\}$ and $E(T') =\{(v,v)\}.$ By repeated use of Proposition  \ref{fold}, we get that $\text{Hom}(T,G_{1,X})\simeq \text{Hom}(T',G_{1,X})$. 

As $T'$ is a single looped vertex and the graph $G_{1,X}$ has a loop at every vertex, we get Hom$(T',G_{1,X})$ $\simeq \text{Cl}(G_{1,X})$ (see Definition \ref{def:Cl}). Further, $\text{Cl}(G_{1,X}) = \text{Sd}(X)$, the first barycentric subdivision of $X$. Therefore, Hom$(T',G_{1,X})\simeq \text{Sd}(X) \simeq X$.

\item $T$ has no loop at any vertex. 

In this case, Equation \eqref{diamcond} implies that $T= K_{n}$, the complete graph on $n \geq 2$ vertices $V(T)= \{1, \ldots , n\}$. Hence, $E(T) =\{(i,j) \ | \ i \neq j \ \text{and} \ i,j  \in \{1, 2, \ldots , n\}\}$. 
Consider $\eta \in \text{Hom}(T,G_{1,X})$.  For each $i \in \{1, \ldots , n\}$, $\eta(i) \subseteq V(G_{1,X})$, {\it i.e.}, any element in $\eta(i)$ is a simplex in $X$.  As $(i,j) \in E(T)$ for any two distinct elements $i$ and $j$ in $V(T)$, from Equation \eqref{homcondition}, we have \begin{equation}\label{eq4.1}
\si \in \eta(i) ~ \text{and}~ \tau \in \eta(j) \Rightarrow \ \text{either}\ \si \subseteq \tau \ \text{or} \ \tau \subseteq \si.
\end{equation}
In other words, $\eta(i) \subseteq N(\eta(j))$, $\forall \ i \neq j$.
For each $i \in \{1, \ldots , n\}$, the elements in $\eta(i)$ are arranged as follows: 
\begin{equation}\label{eta}
 \eta (i)=\{\si^i_1, \ldots, \si^i_{p_i}\} \ \text{where}\ \text{dim}(\si_{j}^{i}) \leq \text{dim}(\si_{j+1}^{i}),\ \forall \ j \in \{1, \ldots , p_{i}-1\}. 
\end{equation} 

So, for each $i \in \{1, \ldots , n\}$, one of the following statements holds. There exists

\begin{itemize} 
\item[(i)] $ \alpha \in \mathbb{N} \cup \{0\}$ such that dim$(\si)=\alpha~ \forall ~\si \in \eta(i)$ or
\item[(ii)] $ i' \in \{1, \ldots , p_i-1\}$, with $\text{dim}(\si_{i'}^{i}) < \text{dim}(\si_{i'+1}^{i})$ and $\text{dim}(\si_{j}^{i}) = \text{dim}(\si_{i'}^{i})$, $\forall \ j \in \{1, \ldots , i'\}$.
\end{itemize}
Define 
\begin{eqnarray}
O(\eta(i)) & =\text{min}\ \{\text {dim}(\si^i_j) \ | \ j \in \{1, \ldots ,p_i\}\} \label{a12} \\ 
O(\eta) & = \text{min} \ \{O(\eta(i)) \ | \ i \in \{1, \ldots ,n\}\}.\label{b12}
\end{eqnarray}
Considering the neighbors of $\eta(i)$, for each $i \in \{1, \ldots , n\}$, we prove the following.
\begin{prop}\label{prop4.1}
If $\eta \in \text{Hom}(K_n,G_{1,X})$, then $\bigcap\limits_{i=1}^{n} N(\eta(i)) \neq \emptyset$.
\end{prop}
\begin{proof}
 From Equation \eqref{b12}, there exists $ k \in \{1, \ldots , n\}$ such that $O(\eta)=O(\eta(k))$. As in Equation \eqref{eta}, let $\eta(k)=\{\si^k_1, \ldots, \si^k_{p_k}\}$,  with $\text{dim}(\si_{j}^{k} ) \leq \text{dim}(\si_{j+1}^{k} )$, $\forall \  j \in \{1, \ldots , p_{k} -1\}$. Further, from Equation \eqref{a12}, there exists $\ell \in \{1, \ldots , p_k\},$ such that $\text{dim}(\si^k_{\ell})= O(\eta(k))$. 
 
Define, 
\begin{equation*}i_1= \text{max}\{t \ | \ t \in \{1, \ldots , p_k\} \text{ and } \text{dim}(\si^k_{t})= O(\eta)=O(\eta(k)\}.
\end{equation*}
Thus, $\forall \ j \in \{1, \ldots , i_1\}$, $\text{dim}(\si^k_{j})= O(\eta).$  Further, from Equation \eqref{eq4.1}, we see that
 \begin{equation}\label{eq4.5}
 \text{for each} \ \si \in \eta(i), \ i \neq k, \ \si_{j}^{k} \subseteq \si, \ \forall \ j \in \{1, \ldots , i_1\}.
 \end{equation}
 Let $\tau_{0} =  \cup \{\si_{j}^{k} \ | \ j \in \{1, \ldots , i_1\}\}.$  As $n \geq 2$, we can choose $ t \in V(T) \sm \{ k\}.$  From Equation \eqref{eq4.5}, $\tau_{0} \subseteq \si $, $\forall \ \si \in \eta(t)$. Thus, $\tau_{0}$ is a simplex in $X$ and hence $\tau_{0} \in V(G_{1,X})$. Further, $\si_{j}^{k}\subseteq \tau_{0}$, $\forall ~j \leq i_1$ implies that $\tau_{0} \in N(\si_{j}^{k})$, for each such $j$.  

Now, consider $i \in \{1,2,\ldots,i_1\}$ and $\tau \in \eta(j)$ for $j \neq k$. Since dim$(\si_i^k)\leq \text{dim}(\tau)$, using Equation \eqref{eq4.1}, $\si_i^k \subseteq \tau$. This implies that $\cup \{\si_{j}^{k} \ | \ j \in \{1, \ldots , i_1\}\} \subseteq \tau$, {\it i.e.}, $\tau_0 \subseteq \tau $. Therefore, $\tau_{0} \in N(\eta(i))$, $\forall \ i \neq k$. Define
 \begin{eqnarray}\label{eq4.6}
\s_{0} = \{\si_{i_1+1}^k, \ldots , \si_{p_{k}}^{k}\} \ \text{and} \  & \s_{1} = \{ \si \in \s_{0} \ | \ \tau_{0} \not\subset \si \ \text{and } \ \si \not\subset \tau_{0}\}.
 \end{eqnarray}
 
 We consider the following possibilities for $\s_1$.
 \begin{enumerate}
 \item $\s_{1} =\emptyset$.  
 
 Here $\tau_{0} \in N(\si)$, $\forall \ \si \in \s_{0}$. Thus, $\tau_{0} \in N(\eta(k))$, thereby proving Proposition \ref{prop4.1}. 
 
 \item $\s_{1} \neq \emptyset.$ 
 
 Define $\tau_{1} =\tau_{0} \cup \{\si \ | \ \si \in  \s_{1}\}$. 
 
 If $\s_{1}=\s_{0}$, then $\si \subseteq \tau_{1}$, $\forall \ \si \in \eta(k)$. Thus $\tau_{1} \in N(\eta(k))$. Moreover, using the same argument as in the case of $\tau_0$, we see that $\tau_{1} \in N(\eta(i))$, $\forall \ i \neq k$. Therefore, $\tau_{1} \in N(\eta(i))$, $\forall \ i \in \{1, \ldots , n\}$.
 
 Now, consider the case where $\s_{0} \sm \s_{1} \neq \emptyset$ and $\s_{0} \sm \s_{1} \subsetneq \s_{0}$. Inductively, for $\ell \in \{2, \ldots , p_{k}-i_{1}-1\}$, define 
 \begin{eqnarray}
\s_{\ell} & = & \{ \si \in \s_{0} \sm \{\s_{1} \cup \ldots \cup \s_{\ell-1}\} \ | \ \tau_{\ell-1} \not\subset \si \ \text{and } \ \si \not\subset \tau_{\ell -1}\},
\end{eqnarray}
where $\tau_{\ell-1}  =  \tau_{\ell-2} \ \cup \ \{\si \ | \ \si \in \s_{\ell-1}\}.$

  If $\s_{\ell} = \emptyset$  or $\s_{0} = \s_{1} \cup \ldots  \ \cup \  \s_{\ell}$ for some $\ell \in \{1, \ldots , p_{k}-i_1-1\}$, then by the same argument as that in the case $\ell=1$, we get a proof of Proposition \ref{prop4.1}.

So, consider the case when  $\s_{l} \neq \emptyset$ and $\s_{0} \neq \s_{1}\cup \ldots \cup \s_{\ell}$ for all $\ell \in \{1, \ldots , p_{k}-i_1 -1\}$. 

Let  $\s_{p_k - i_1} =  \{ \si \in \s_{0} \sm \{\s_{1} \cup \ldots \cup \s_{p_k-i_1-1}\} \ | \ \tau_{p_{k}-i_1 -1} \not\subset \si \ \text{and } \ \si \not\subset \tau_{p_{k} -i_1 -1}\}.$ 

 Here, $\s_{p_{k}-i_1}$ has exactly one element and $\s_{0} =\s_{1}\cup \ldots \cup \s_{p_{k}-i_1}.$ So, $\tau_{p_k-i_1}$, which is $\tau_{p_k-i_1-1}\cup \{\si \ | \ \si \in \s_{p_k -i_1}\}$, is a neighbor of each element of $\eta(k)$. 

Let $i \in \{1,2,\ldots, n\}$ and $i \neq k$. Using induction, we have $\tau_{p_{k}-i_1-1} \in N(\eta(i))$. In particular, $\tau_{p_{k}-i_1-1} \subseteq \tau$ for each $\tau \in \eta(i)$. Let $\si \subseteq \tau_{p_{k}-i_1} $ and $\si \not\subset \tau_{p_{k}-i_1-1}$. Since $\si \in \s_{p_k-i_1} \subset \eta(k)$, from Equation \eqref{eq4.1}, $\si \subseteq \tau$ for each $\tau \in \eta(i)$. This implies that $\tau_{p_{k}-i_1} \subseteq \tau$ for each $\tau \in \eta(i)$. Thus, $\tau_{p_{k}-i_1} \in N(\eta(i))$.

As $\tau_{p_{k}-i_1} \in N(\eta(i))$, $\forall \ i \in \{1, \ldots , n\}$, we get the proof of Proposition \ref{prop4.1}.
\end{enumerate}
\end{proof}
\end{enumerate}

Using Proposition \ref{prop4.1}, we prove the main result of this article.

\begin{proof}[Proof of Theorem \ref{thm2}:]~~~~

\medskip

Let  $\P =\F(\text{Hom}(K_n,G_{1,X}))$ and $\Q =\F(\text{Hom}(K_{n-1},G_{1,X}))$. Define $\varphi :\P \longrightarrow \Q$ by $\varphi(\eta)=\eta |_{\{1, \ldots , n-1\}}$, {\it i.e.}, $\varphi (\eta) :\{1, 2, \ldots , n-1\} \rightarrow \Q$  is defined as $\varphi (\eta) (i) =\eta(i)$, $\forall \ i \in \{1, 2, \ldots , n-1\}$.

First, let $\rho \in \Q$, {\it i.e.}, $\rho : \{1, \ldots ,n-1\} \rightarrow 2^{V(G_{1,X})}\sm \{\emptyset\}$ such that $\rho(i)\times \rho(j) \subseteq E(G_{1,X})$, for $i\neq j$. We see that $\varphi^{-1}(\rho)$ is the set of all $n$ tuples $(\rho(1),\ldots, \rho(n-1),B)$, where $B \subseteq V(G_{1,X})$ such that for all $i \in [n-1]$, $\rho(i) \times B \subseteq E(G_{1,X})$. 

Define, $\eta_1: [n] \longrightarrow 2^{V(G_{1,X})}\sm \{\emptyset\}$ by 
\begin{equation}\label{definitioneta1}
 \eta_1(i)= \left\{\def\arraystretch{1.2}%
  \begin{array}{@{}c@{\quad}l@{}}
    \rho(i) & \text{if $i \in \{1,\ldots,n-1\}$, }\\
    \bigcap\limits_{i=1}^{n-1} N(\rho(i)) & \text{if $i=n$.}\\
  \end{array}\right.
\end{equation}

Since $n \geq 3$ and $\rho \in \text{Hom}(K_{n-1},G_{1,X})$, using Proposition \ref{prop4.1}, we have $\bigcap\limits_{i=1}^{n-1} N(\rho(i)) \neq \emptyset$. Thus, $\eta_1(n) \neq \emptyset$. Further, for each $i \in \{1,\ldots,n-1\}$, $\bigcap\limits_{j=1}^{n-1} N(\rho(j)) \subseteq N(\rho(i)) = N(\eta_1(i))$. This implies that $\eta_1(n) \subseteq N(\eta_1(i))$ for all $i \in \{1,\ldots,n-1\}$. Therefore, $\eta_1 \in \text{Hom}(K_n,G_{1,X})$ and $\varphi(\eta_1)=\rho$.

Moreover, if $\gamma \in \varphi^{-1}(\rho)$, then for each $i \in \{1,\ldots,n-1\}$, $\gamma(n)\subseteq N(\gamma(i))$ and $\gamma(i)=\rho(i)=\eta_1(i)$. This implies that $\gamma(n)\subseteq N(\rho(i))$ for all $i \in \{1,\ldots,n-1\}$. Thus, $\gamma(n)\subseteq \bigcap\limits_{i=1}^{n-1} N(\rho(i))=\eta_1(n)$ implying that $\gamma \leq \eta_1$. Therefore, $\eta_1$ is a maximal element of $\varphi^{-1}(\rho)$, implying that $\Delta(\varphi^{-1}(\rho))$ is a cone and therefore contractible.

Now consider, $\rho \in \Q$ and $\eta \in \P$, such that $\rho \leq \varphi(\eta)$, {\it i.e.}, $\rho(i)\subseteq \eta(i)$ for all $i \in [n-1]$. Clearly, $\bigcap\limits_{i=1}^{n-1} N(\rho(i)) \supseteq \bigcap\limits_{i=1}^{n-1} N(\eta(i)) \supseteq \eta(n)$. Then, $$\varphi^{-1}(\rho)\cap \P_{\leq \eta}= \{(\rho(1),\ldots,\rho(n-1),B) \ | \ B \subseteq \eta(n), B\neq \emptyset\}.$$ Since $\eta(n) \subseteq \bigcap\limits_{i=1}^{n-1} N(\rho(i))$, $(\rho(1),\ldots,\rho(n-1),\eta(n)) \in \text{Hom}(K_n,G_{1,X})$ and $(\rho(1),\ldots,\rho(n-1),\eta(n))$ is a maximal element of the poset $\varphi^{-1}(\rho)\cap \P_{\leq \eta}$. 

Thus, the map $\varphi$ satisfies both the conditions of Lemma \ref{fibthm}. Therefore, the induced map $\Delta(\varphi)$ is a homotopy equivalence. 
Hence, the barycentric subdivision of Hom$(K_n,G_{1,X})$ is homotopy equivalent to the barycentric subdivision of $\text{Hom}(K_{n-1},G_{1,X})$, {\it i.e.}, $\text{Sd}(\text{Hom}(K_n,G_{1,X})) \simeq \text{Sd}(\text{Hom}(K_{n-1},G_{1,X}))$. This proves Theorem \ref{thm2}. 
\end{proof}

We are now in a position to give an alternate proof of Theorem \ref{thm3}.

\begin{proof}[Proof of Theorem \ref{thm3}:]~~~~

\medskip

If $T$ has a looped vertex, say $v$, then $T$ folds onto a subgraph $T'$, where $V(T')=\{v\}$ and $E(T')=\{(v,v)\}$. Using Proposition  \ref{fold}, we get $\text{Hom}(T,G_{1,X})\simeq \text{Hom}(T',G_{1,X})$. Observe that,  Hom$(T',G_{1,X})\simeq \text{Cl}(G_{1,X})$.  As the clique complex of $G_{1,X}$, $\text{Cl}(G_{1,X})$  and $\text{Sd}(X)$ are identical and $ \text{Sd}(X)$ is homeomorphic to $X$, we get $\text{Hom}(T,G_{1,X})\simeq X$.  

Now, consider the case where $T$ does not have a loop at any vertex.
Here, $T=K_n$, the complete graph on $n$ vertices for some $n\geq 2.$  From Theorem \ref{thm2}, $\text{Hom}(K_n,G_{1,X})\simeq \text{Hom}(K_{n-1}, G_{1,X})$ for $n\geq 3$. 

Recursively, for any $n >2$,  $\text{Hom}(K_n,G_{1,X})\simeq \text{Hom}(K_2, G_{1,X})$. 

From Theorem \ref{ng1x}, we have $\N(G_{1,X}) \simeq X$. Now, using Lemma \ref{homtonbd}, we get $\text{Hom}(K_2,G_{1,X})\simeq X$. Therefore, $\text{Hom}(K_n,G_{1,X})\simeq X,$ for all $n \geq 2$. 

This completes the proof of Theorem \ref{thm3}.
\end{proof}

 \subsection*{Acknowledgements}
  The author would like to thank Nandini Nilakantan for several helpful discussions and kind advices.



\bibliographystyle{alpha}

\end{document}